\documentclass[11pt]{article}
\usepackage{cite}
\usepackage{mathrsfs}
\usepackage{amsmath, amssymb, amsthm, latexsym, amsfonts, epsfig, color,graphicx}
\usepackage{setspace,bm}
\usepackage{hyperref}

\allowdisplaybreaks[4]
\numberwithin{equation}{section}
\begin{document}

	\newcommand{\s}{\sigma}
	\renewcommand{\k}{\kappa}
	\newcommand{\p}{\partial}
	\newcommand{\D}{\Delta}
	\newcommand{\om}{\omega}
	\newcommand{\Om}{\Omega}
	\renewcommand{\phi}{\varphi}
	\newcommand{\e}{\epsilon}
	\renewcommand{\a}{\alpha}
	\renewcommand{\b}{\beta}
	\newcommand{\N}{{\mathbb N}}
	\newcommand{\R}{{\mathbb R}}
	\newcommand{\eps}{\varepsilon}
	\newcommand{\EX}{{\mathbb{E}}}
	\newcommand{\PX}{{\mathbb{P}}}
	
	\newcommand{\cF}{{\cal F}}
	\newcommand{\cG}{{\cal G}}
	\newcommand{\cD}{{\cal D}}
	\newcommand{\cO}{{\cal O}}
	
	\newcommand{\de}{\delta}

	\newcommand{\grad}{\nabla}
	\newcommand{\n}{\nabla}
	\newcommand{\curl}{\nabla \times}
	\newcommand{\dive}{\nabla \cdot}
	
	\newcommand{\ddt}{\frac{d}{dt}}
	\newcommand{\la}{{\lambda}}
	
	\newcommand{\dif}{\mathop{ }\!\mathrm{d}}

	\newtheorem{theorem}{Theorem}[section]
	\newtheorem{lemma}{Lemma}[section]
	\newtheorem{remark}{Remark}[section]
	\newtheorem{example}{Example}[section]
	\newtheorem{definition}{Definition}[section]
	\newtheorem{corollary}{Corollary}[section]
	\newtheorem{assumption}{Assumption}[section]
	\newtheorem{prop}{Proposition}[section]
	\newtheorem{notation}{Notation}[section]
	\def\proof{\mbox {\it Proof.~}}
	\makeatletter
	\@addtoreset{equation}{section}
	\makeatother
	\renewcommand{\theequation}{\arabic{section}.\arabic{equation}}

	\makeatletter\def\theequation{\arabic{section}.\arabic{equation}}\makeatother
	
	\title{ Criteria for asymptotic stability of eventually continuous Markov-Feller semigroups}
	
	\author{
		{ \bf\large Ting Li, Xianming Liu\footnote{Corresponding author: xmliu@hust.edu.cn}}\hspace{2mm}
		\vspace{1mm}\vspace{1mm}\\
		{\it\small  School of Mathematics and Statistics},\\
		{\it\small Huazhong University of Science and Technology},
		{\it\small Wuhan 430074,  China}}

	\date{\today }
	\maketitle
	
	\begin{abstract}
		We establish two criteria for the asymptotic behavior of Markov-Feller semigroups. First, we present a criterion for convergence in total variation to a unique invariant measure, requiring only $TV$-eventual continuity of the semigroup at a single point. Second, we propose new asymptotic stability criteria that satisfy two non-uniform asymptotic conditions and eventual continuity at a single point. Furthermore, we provide an explicit example where eventual continuity at a point admits simple verification, while proving the corresponding global property requires significantly more sophisticated methods.
	\end{abstract}
	
	

	
	\setcounter{secnumdepth}{5} \setcounter{tocdepth}{5}
	
	\makeatletter
	\newcommand\figcaption{\def\@captype{figure}\caption}
	\newcommand\tabcaption{\def\@captype{table}\caption}
	\makeatother

	\section{Introduction}
	Markov processes, crucial research topics in stochastic systems, are widely applied in finance, biology, and engineering due to their memoryless property. A key aspect of their long-term behavior is the convergence of the associated Markov semigroups to a unique invariant measure, which is crucial to understanding the system's stability.
	
	The long-term behavior of Markov semigroups has recently attracted considerable interest, as referenced \cite{Bess, Gong1, Gong2, unique, epro, Laso1, Feller, stability, Tom1}. For diffusion processes, the uniqueness of invariant measures in total variation can typically be established via two distinct approaches: either by applying Harris' theorem (see \cite{harris}) directly or by combining the existence of an invariant measure with Doob's theorem (see \cite{Da}) to derive uniqueness. In the case of SPDEs driven by degenerate Wiener noise, Hairer and Mattingly (see \cite{Hairer-1}) introduced an alternative framework based on the asymptotic strong Feller property---a generalization of the classical strong Feller condition---which ensures uniqueness under weaker regularity assumptions.
	Recent efforts have identified conditions under which the uniqueness of invariant measures for Markov-Feller semigroups can be established.
	For example, Szarek \cite[Theorem 3.3]{Feller} established a sufficient condition for the asymptotic stability of Markov-Feller semigroups under the $e$-property assumption. Szarek and Worm \cite{Tom1} further demonstrated a necessary and sufficient condition that is both necessary and sufficient for the asymptotic stability of Markov-Feller semigroups. Their result is as follows: if $\{P_t\}_{t\geq 0}$ has $e$-property on $E$,
	\begin{equation*}\label{a}
		\forall \epsilon>0, \inf_{x\in E}\liminf_{t\rightarrow\infty} P_t(x, B(z,\epsilon))>0 \ \   \bm{\Leftrightarrow}\ \  \text{asymptotically stable on $E$.} \tag{A}
	\end{equation*}
	The $e$-property is also known as equicontinuity (see \cite{epro}), i.e., for any Lipschitz bounded function $f$, the family of functions $\{P_tf\}_{t\geq 0}$ is uniformly continuous at any point of a Polish space $E$.

	Non-equicontinuous Markov semigroups have attracted research attention. Specifically, the asymptotic stability of Markov-Feller semigroups under eventual continuity (see Definition \ref{def1}) has been investigated in \cite{Gong1, Gong2, Jaroszewska}. Subsequently, Gong and co-authors \cite{Gong3} demonstrated that asymptotic stability can be achieved if Markov-Feller semigroups satisfy eventual continuity at a single point and a uniform lower bound, i.e. 
		\begin{equation}\label{b}
			\left.
			\begin{array}{rr}
				\exists z\in E, \  \{P_t\}_{t\geq 0} \ \text{is} \  \text{eventually continuous at $z$}    \\
				\forall x\in E, \forall \epsilon>0, \inf\limits_{x\in E}\liminf\limits_{t\rightarrow\infty} P_t(x, B(z,\epsilon))>0
			\end{array}
			\right\}
			\bm{\Leftrightarrow} \ \ \text{asymptotically stable on $E$.} \nonumber \tag{B}
	\end{equation}
	There are many instances of Markov semigroups that fail to meet the condition of equicontinuity. For example, see \cite[Example 4.1, Example 4.2]{Gong3}, \cite[Example 7]{Jaroszewska} and \cite[Example 4]{Kulik}. A toy model that can be seen as a special case of \cite[Proposition 5.2]{Gong3} reads
	\begin{align}\label{langevin}
		\dif X_t=(\frac32X_t-X_t^3) \dif t+ X_t\dif B_t,\ \ \  X_0=x\in \mathbb{R}.
	\end{align}
	We observe that \eqref{langevin} does not satisfy the $e$-property at $0$ and is not eventually continuous at $0$. Specifically, $0$ is an unstable equilibrium for \eqref{langevin}.

	Parallel studies are actively exploring another aspect of the long-term behavior of Markov-Feller semigroups. For instance, Komorowski and co-authors \cite{epro} studied the local weak* mean ergodicity of Markov semigroups, specifically defined on the set $\mathcal{T}$  (see \eqref{defT}). Based on this, Szarek and the coauthors \cite{stability} established asymptotic stability on the same set $\mathcal{T}$:
	\begin{equation}\label{c}
		\left.
		\begin{array}{rrr}
			\{P_t\}_{t\geq 0} \ \text{has} \    e\text{-property on $E$}\\
			\exists z\in E, \forall x\in E, \forall \epsilon>0,\liminf\limits_{t\rightarrow \infty} Q_t(x, B(z,\epsilon))>0\\
			\liminf\limits_{t\rightarrow\infty} P_t(z, B(z,r))>0
		\end{array}
		\right\}
		\bm{\Rightarrow}\ \  \text{asymptotically stable on $\mathcal{T}$ .} \tag{C}
	\end{equation}
	
	In this paper, we investigate the necessary and sufficient conditions for the ergodicity of Markov-Feller semigroups under more general conditions, particularly focusing on local eventually continuous ones. Firstly, we provide the equivalent condition for the convergence of the semigroup to its unique invariant measure in total variation distance, i.e.,
	\begin{equation}
		\left.
		\begin{array}{rr}
			\exists z\in E, \  \{P_t\}_{t\geq 0} \ \text{is} \  TV\text{-eventually continuous at $z$}    \\
			\forall x\in E, \forall \epsilon>0, \inf\limits_{x\in E}\liminf\limits_{t\rightarrow\infty} P_t(x, B(z,\epsilon))>0
		\end{array}
		\right\}
		\bm{\Leftrightarrow} \ \ ||P_t(x,\cdot)-\mu||_{TV}\rightarrow 0, \ \text{as} \ t\rightarrow \infty \nonumber,
	\end{equation}
	where $TV$-eventually continuous is a uniform strengthening of eventually continuous, requiring convergence uniformly over all bounded measurable functions. Unlike \eqref{a} and \eqref{b}, we have achieved asymptotic stability in the sense of total variation.

	Secondly, based on Theorem 2 from \cite{stability}, we derive an additional equivalent condition for asymptotic stability under the condition $\mathcal{T}=E$.
	The result in Theorem \ref{thm2} can be summarized as follows: let $\mathcal{T}=E$,
	\begin{equation}
		\left.
		\begin{array}{rrr}
			\exists z\in E, \  \{P_t\}_{t\geq 0} \ \text{is}  \ \text{eventually continuous at $z$}    \\
			\forall x\in E, \forall \epsilon>0,\limsup\limits_{t\rightarrow \infty} Q_t(x, B(z,\epsilon))>0\\
			\liminf\limits_{t\rightarrow\infty} P_t(z, B(z,r))>0
		\end{array}
		\right\}
		\bm{\Leftrightarrow}\ \  \text{asymptotically stable on $E$.} \nonumber
	\end{equation}
	Contrasting with the condition outlined in \eqref{c}, we have shifted our requirement from global equicontinuity to local eventual continuity. Furthermore, instead of requiring $z$ to be within the support of all limit measures of $Q_t(x,\cdot)$, we now only require $z$ to be within the support of at least one such limit measure.

	In addition, we illustrate the applicability of the theorems we have presented through two applications. First, we rigorously establish the asymptotic stability for the Navier-Stokes equations driven by pure jump noise for their Markov semigroup.  Second, we present an SDE driven by a multiplicative Poisson process and establish its asymptotic stability. Importantly, while the eventual continuity of its Markov semigroup can be readily shown in local regions, demonstrating it globally is much more intricate.
	
	The structure of this paper is as follows. In Section \ref{sec2}, we introduce the necessary concepts and notations for this paper and present our main results. Subsequently, in Section \ref{section3}, we provide two applications to illustrate the applicability of our theorems. In Section \ref{proofs}, we present the proofs of the main theorems.
	
	\section{Notations and main results}\label{sec2}
	
	Let $(E,d)$  be a Polish space and $\mathcal{M}(E)$ be the set of all probability measures on $E$. We introduce the following notations. For a Polish space $E$ we denote respectively by $B_b(E)$, $C_b(E)$ and $L_b(E)$ the space of bounded and measurable functions, the space of continuous and bounded functions, and the space of Lipschitz and bounded functions on $E$. Let $\mu\in\mathcal{M}(E)$, we denote the support of the measure $\mu$ by
	\begin{align*}
		\text{supp}\ \mu=\{x\in E: \mu(B(x, \epsilon))> 0, \forall \epsilon >0\},
	\end{align*}
	where $B(x, \epsilon)=\{y\in E: d(x,y)<\epsilon\}$.
	
	Denote by $\{P_{t}\}_{t\geq0}$ the homogeneous Markov-Feller semigroup. Let $\{P_t^*\}_{t\geq 0}$ be the dual semigroup defined on $\mathcal{M}(E)$ by the formula $P_t^*\mu(\cdot):=\int_E P_t\mathbf{1}_{\cdot} \mathrm{d}\mu$ for any $\mu\in\mathcal{M}(E)$. And we define
	\begin{align}\label{defQ}
		Q_t^*\mu:=\frac1t\int_0^t P_s^*\mu \dif s, \ \ \forall \mu \in \mathcal{M}(E).
	\end{align}
	We write $Q_t(x, \cdot)$ in particular case when $\mu=\delta_x$. Let
	\begin{align}\label{defT}
		\mathcal{T}=\{x\in E: \{Q_t(x,\cdot)\}_{t\geq 0} \ \text{is tight}\}.
	\end{align}
	

	\begin{definition}\label{def1}
		A Markov semigroup $\{P_t\}_{t\geq 0} $ satisfies the eventually continuous at $z\in E$, if for every $f\in L_b(E)$,
		\begin{align*}
			\limsup_{x\rightarrow z}\limsup_{t\rightarrow \infty}|P_tf(x)-P_tf(z)|=0.
		\end{align*}
	\end{definition}
	Eventual continuity, an $e$-property extension, was initially presented in \cite{Gong1} and \cite{Jaroszewska}, with \cite{Jaroszewska}  referring to it as the asymptotic equicontinuity condition.
	
	Now, we give a new version of the {$TV$-eventual continuity}.
	\begin{definition}\label{defB-e}
		A Markov semigroup $\{P_t\}_{t\geq 0}$ satisfies the {$TV$-eventually continuous} at $z\in E$, if
		\begin{align*}
			\limsup_{x\rightarrow z}\limsup_{t\rightarrow \infty}\sup_{||f||_{\infty}= 1} |P_tf(x)-P_tf(z)|=0.
		\end{align*}
	\end{definition}
	
	Before presenting the main results, let us recall the following definitions.

	\begin{definition}
		The total variance distance of two probability measures $\mu_1$, $\mu_2\in \mathcal{M}(E)$ is defined by
		\begin{align*}
			||\mu_1-\mu_2||_{TV}=\frac12\sup_{||f||_\infty = 1}\Big|\int_{E} f(x)\mu_1(dx)-\int_{E} f(x)\mu_2(dx) \Big|.
		\end{align*}
	\end{definition}

	\begin{definition}\label{defas}
		A Markov semigroup $\{P_t\}_{t\geq 0}$  is {asymptotically stable}, if there exists a measure $\mu\in\mathcal{M}(E)$ such that
		\begin{align*}
			\lim_{t\rightarrow \infty} P_t^*\nu=\mu, \ \ \forall \nu\in\mathcal{M}(E).
		\end{align*}
	\end{definition}

	The first result is stated in the following theorem.
	
	\begin{theorem}\label{thm0}
		Let $\{P_t\}_{t\geq 0}$ be a Markov-Feller semigroup on $E$. Then the following two statements are equivalent:\\
		(1) $\{P_t\}_{t\geq 0}$ has a unique invariant measure $\mu$ and satisfies for any $\nu\in \mathcal{M}(E)$,
		\begin{align}\label{TV}
			\lim_{t\rightarrow\infty} ||P_t^*\nu-\mu||_{TV}=0.
		\end{align}
		(2) There exists $z\in E$ such that $\{P_t\}_{t\geq 0}$ is $TV$-eventually continuous at $z$ and for all $\epsilon >0$
		\begin{align} \label{C4}
			\inf_{x\in E}\liminf_{t\rightarrow\infty} P_t(x, B(z,\epsilon))>0.
		\end{align}
	\end{theorem}
	
	Replacing the $TV$-eventually continuous hypothesis in $(2)$ with eventually continuous and following similar arguments as in Theorem \ref{thm0}, one has the next propostion, which was previously established in \cite[Theorem 3.16]{Gong3}:
	\begin{prop}\label{pro1}
		Let $\{P_t\}_{t\geq 0}$ be a Markov-Feller semigroup on $E$. Then the following two statements are equivalent:\\
		(1) $\{P_t\}_{t\geq 0}$ is asymptotic stability with unique invariant measure $\mu$;\\
		(2) There exists $z\in E$ such that $\{P_t\}_{t\geq 0}$ is eventually continuous at $z$ and for all $\epsilon >0$ \eqref{C4} is hold.
	\end{prop}

Inspired by \cite[Proposition 7.1]{phdth}, we derive a sufficient condition for the asymptotic stability of Markov-Feller semigroups.  
	
	\begin{corollary}\label{col1}
		Let $\{P_t\}_{t\geq 0}$ be a Markov-Feller semigroup on $E$ associated with a  Feller process $\{X_t\}_{t\geq 0}$. Assume that there exists $z\in E$ such that $\{P_t\}_{t\geq 0}$ is eventually continuous at $z$. And there exists $\tilde{r}>0$, for all $\epsilon>0$, $r>\tilde{r}$, there exists $T>0$ such that
		\begin{align*}
			\inf_{x\in B(z,r)}P_{T}(x, B(z,\epsilon))>0.
		\end{align*}
		Furthermore, there exists $p> 0$, $b>0$ and $\kappa>0$ such that for every $x\in E$
		\begin{align*}
			\EX [d(X_t^x,z)]^p\le \rho_x(t)+b\mathbf{1}_{B(z,\kappa)},
		\end{align*}
		where $\rho_x:(0,\infty)\rightarrow (0,\infty)$ is a nonincreasing function and $\lim\limits_{t\rightarrow \infty}\rho_x(t)=0$, for all $x\in E$.   Then, $\{P_t\}_{t\geq 0}$ is asymptotically stable.
	\end{corollary}
	
	Next, we provide a different set of conditions for asymptotic stability.

	\begin{theorem}\label{thm2}
		Let $\{P_t\}_{t\geq 0}$ be a Markov-Feller semigroup on $E$, and $\mathcal{T}=E$. Then the following two statements are equivalent:\\
		(1) $\{P_t\}_{t\geq 0}$ is asymptotic stability with unique invariant measure $\mu$;\\
		(2) There exists $z$ such that $\{P_t\}_{t\geq 0}$ is eventually continuous at $z$ and for all $x\in E$, $\epsilon >0$,
		\begin{align}
			\limsup_{t\rightarrow\infty} Q_t(x, B(z,\epsilon))>0, \label{C1}\\
			\liminf_{t\rightarrow\infty} P_t(z, B(z,\epsilon))>0\label{C2}.
		\end{align}
	\end{theorem}
	
	In fact, if $\mathcal{T}=E$, $\{P_t\}_{t\geq 0}$ is eventually continuous at $z$ and for all $x\in E$, $\epsilon >0$, \eqref{C1} is holds, then $\{P_t\}_{t\geq 0}$ is weak* mean ergodic(see \cite[Definition 2.2]{epro}). This can be derived from the first and second steps in the proof of Theorem \ref{thm2}.

	\section{Application}\label{section3}
	In this section, we present two applications. First, we demonstrate the ergodicity of the stochastic Navier-Stokes equation with degenerate pure jump noise. Second, we establish the asymptotic stability of a class of SDEs with multiplicative Poisson noise.

	\subsection{2D Navier-Stokes equations with degenerate L\'evy noise revisited}
	
	In this subsection, we demonstrate the uniqueness of an invariant measure for stochastic Navier-Stokes equations. Although there are a lot of conclusions regarding the invariant measures of the two-dimensional Navier-Stokes equation, our research establishes the asymptotic stability of these equations. Moreover, we no longer need to rely on tightness to establish the existence of invariant measures, this existence has become a byproduct of our study.

	In \cite{Hairer-1}, the authors proved the uniqueness of the invariant measure for the 2D stochastic Navier-Stokes equation driven by degenerate Brownian motion, using the asymptotic strong Feller property and a weak form of irreducibility. Recently, the ergodicity for stochastic 2D Navier-Stokes equations driven by a highly degenerate pure jump L\'evy noise has been studied in \cite{NSdbL}.   In \cite{Gong1}, the authors have established that the Navier-Stokes equation, driven by degenerate Brownian motions, satisfies eventual continuity and that condition \eqref{C4} holds.

	As an application, this paper revisits the ergodicity of the two-dimensional Navier-Stokes equation with degenerate L\'evy noise.

	Consider the two-dimensional, incompressible Navier-Stokes equation on the torus $\mathbb{T}^2=[-\pi, \pi]^2$ driven by degenerate noise. Denote by
	\begin{align}\label{space}
		H:=\{u(x)\in L^2(\mathbb{T}^2;\mathbb{R}^2): \ \int_{\mathbb{T}^2}  u(x) \dif x=0\}.
	\end{align}
	Let $W_{S_t}=(W_{S_t}^1, W_{S_t}^2, W_{S_t}^3, W_{S_t}^4)$ is a $4$-dimensional subordinated Brownian motion, where $S_t$ is a subordinator process with measure $\nu_S$ satisfying
	\begin{align*}
		\int_0^\infty (e^{\zeta u}-1) \nu_S(\dif u)<\infty \ \ \ \text{for some}\ \ \  \zeta>0,\ \ \ \text{and}\ \ \ \nu_S((0,\infty))=\infty.
	\end{align*}
	The Navier-Stokes equation is given by
	\begin{align} \label{NS}
		{\dif}{u}_t + \big[A{u}_t + B({u}_t, {u}_t)\big] {\dif}t=Q\dif {W}_{S_t}, \ \  u_0=x\in H,
	\end{align}
	where $A$ is a positive self-adjoint operator, $B$ is a bilinear continuous mapping, and $Q:\mathbb{R}^4\rightarrow H$ is a linear operator, satisfying
	\begin{align*}
		QW_{S_t}= q_1\sin(x_1)W_{S_t}^1+q_2\cos(x_1)W_{S_t}^2+q_3\sin(x_1+x_2)W_{S_t}^3+q_4\cos(x_1+x_2)W_{S_t}^4.
	\end{align*}
	Here $q_1,\cdots, q_4$ are non-zero constants.
	For further requirements on the coefficient and noise of this equation, please refer to \cite[Section 1.2]{NSdbL}. Therefore, applying Theorem~\ref{thm2} to \eqref{NS}, we have

	\begin{prop}
		Assume that $\{P_t\}_{t\geq 0}$ is the Markov-Feller semigroup associated with \eqref{NS}. Then, $\{P_t\}_{t\geq 0}$ is asymptotically stable.
	\end{prop}

	\begin{proof}
		From \cite[Proposition 1.4]{NSdbL}, we deduce that $\{P_t\}_{t\geq 0}$ is eventually continuous on $H$. From \cite[Proposition 1.5; Lemma 2.1]{NSdbL}, it follows that equation \eqref{C4} holds.
		
		Indeed, by \cite[Proposition 1.4]{NSdbL}, for any $f\in L_b^1(H)$, $\epsilon>0$ and $x\in H$, there exists a $\delta>0$ such that
		\begin{align}\label{e-p}
			| P_t f(x)-P_tf(y)|\le \epsilon, \ \  \forall t>0 \ \ \text{and}\ \ y \ \text{with}\  ||y-x||_{H}<\delta.
		\end{align}
		This implies that $\{P_t\}_{t\geq 0}$ associated with \eqref{NS} is eventually continuous on $H$.
		
		On the other hand, for a fixed $R>0$ (its value will be provided later on), using \cite[Lemma 2.1]{NSdbL}, and the Chebyshev-Markov inequality we have
		\begin{align*}
			P_t(x, B(0, R))=\EX\mathbf{1}_{\{||u_t^x||_{H}<R\}}=1-\PX(||u_t^x||_{H}\geq R)\geq 1-\dfrac{\EX||u_t^x||_{H}^2}{ R^2}
			\geq  1- \dfrac{e^{-t}||x||_{H}^2+C_1}{R^2},
		\end{align*}
		where $C_1$ is a constant depending on linear operator $Q$,  the L\'evy measure $\nu_{S}$ of pure jump subordinator $S_t$ and, the dimension $d$.
		Let $R>\sqrt{2(||x||_{H}^2+C_1)}$, then
		\begin{align*}
			P_t(x, B(0,R))>\frac12,\ \ \forall t> 0.
		\end{align*}
		From \cite[Proposition 1.5]{NSdbL}, for  any $\epsilon>0$ and $R>\sqrt{2(||x||_{H}^2+C_1)}$, there exists $T>0$ and $p^*>0$, such that $\inf\limits_{x\in B_R} P_T(x, B(0,\epsilon))> p^*>0$. Then, for any $\epsilon>0$ and $x\in H$ we have
		\begin{align}\label{LB}
			\liminf_{t\rightarrow\infty} P_t(x, B(0,\epsilon))&=    \liminf_{t\rightarrow\infty} P_{t+ T}(x, B(0,\epsilon))=\liminf_{t\rightarrow\infty}  \int_{H} P_t(x, \dif y) P_T(y, B(0,\epsilon)) \nonumber \\
			&\geq\liminf_{t\rightarrow\infty}  \int_{B(0,R)} P_t(x, \dif y) P_T(y, B(0,\epsilon))\nonumber \\
			&\geq \liminf_{t\rightarrow\infty} P_t(x, B(0, R)) \inf_{x\in B_R} P_T(x, B(0,\epsilon))\nonumber \\
			&\geq \frac12 \inf_{x\in B_R} P_T(x, B(0,\epsilon))>\frac12 p^*>0.
		\end{align}
		This implies that $\{P_t\}_{t\geq 0}$ associated with \eqref{NS} satisfying conditions \eqref{C1} and \eqref{C2}. Thus, by Theorem \ref{thm2} we conclude that $\{P_t\}_{t\geq 0}$ is asymptotically stable.\bigskip\qed\end{proof}

	\begin{remark}
		We can also prove the asymptotic stability of $\{P_t\}_{t\geq 0}$ associated with \eqref{NS} by applying Theorem \ref{pro1} or Corollary \ref{col1}.
	\end{remark}
	
	\subsection{Stability analysis for SDEs with multiplicative Poisson noise}
	We provide a model with non-degenerate, non-Gaussian noise to demonstrate that verifying the eventual continuity of a Markov-Feller semigroup at a specific point is a straightforward and relatively uncomplicated task.   It also highlights some similarities and differences between the smoothing effects of Gaussian and non-Gaussian noises. Let $N_t$ be a Poisson process of intensity $1$, consider the following SDE:
	\begin{align}\label{pola}
		\dif X_t= (aX_t- bX_t^3) \dif t+ \sigma(X_{t-})\dif N_t, \ \ \ X_0=x\in \mathbb{R},
	\end{align}
	where $a,b>0$, $\sigma$ is a function satisfying: \\
	\noindent${(1)}$ $\exists m,M>0$ such that $\ m<\sigma(x)<M, \ \forall x\in\mathbb{R}$;\\
	${(2)}$ $\exists L_{\sigma}>0$ such that $|\sigma(x)-\sigma(y)|\le L_\sigma|x-y|,\  \forall x,y\in\mathbb{R}$.\\
	
	Then, we have the following result.
	\begin{prop}
		Assume that $\{P_t\}_{t\geq 0}$ is the Markov-Feller semigroup associated with \eqref{pola}. Then, $\{P_t\}_{t\geq 0}$ is asymptotically stable.
	\end{prop}
	\begin{proof}
		Fix any $x,y\in \mathbb{R}$, take $X_t^x, X_t^y$ solving \eqref{pola} with initial datum $x,y$. We divide the proof into three steps to prove that $\{P_t\}_{t\geq 0}$ is asymptotically stable.\\
		\textbf{Step 1.} We prove the local eventual continuity of $\{P_t\}_{t\geq 0}$. To this end, we define $\tilde{X}_t^y$ as the solution to the following equation with the initial condition $y$:
		\begin{equation}
			\dif \tilde{X}_t^y=a\tilde{X}_t^y-b\cdot (\tilde{X}_t^y)^3\  \dif t+ \lambda (X_t^x-\tilde{X}_t^y) \dif t + \sigma(\tilde{X}_{t-}^y)\dif N_t,\ \ \ \tilde{X}_0^y=y, \nonumber
		\end{equation}
		where $\lambda$ is a positive constant that will be given later.
		Let $Z_t=X_t^x-\tilde{X}_t^y$ and $A_t= (\tilde{X}_t^y)^2+ \tilde{X}_t^yX_t^x + (X_t^x)^2$, we have
		\begin{align*}
			\dif Z_t=(a-\lambda)Z_t\dif t -bA_tZ_t\dif t+[\sigma(X_{t-}^x)-\sigma(\tilde{X}_{t-}^y)]\ \dif N_t, \ \ \  Z_0=x-y.
		\end{align*}
		By using It\^{o} formula,
		\begin{align*}
			\dif (Z_t)^2=2(a-\lambda-bA_t)(Z_t)^2 \dif t+ \big[(\sigma(X_{t-}^x)-\sigma(\tilde{X}_{t-}^y))^2+2Z_t(\sigma(X_{t-}^x)-\sigma(\tilde{X}_{t-}^y))\big]\ \dif N_t.
		\end{align*}
		Let $\lambda> (a+L_{\sigma})+\frac{L_{\sigma}^2}2$, by $A_t\geq 0$ and comparison principle,
		\begin{align}\label{ineq1}
			\EX|Z_t|^2\le |x-y|^2e^{(2a-2\lambda+2L_{\sigma}+L_{\sigma}^2)t}\rightarrow 0,\ \ \  \text{as}\ \ \ t\rightarrow\infty,\ \ \ x\rightarrow y.
		\end{align}
		Let $\tilde{Z}_t=\tilde{X}_t^y- X_t^y$ and $\tilde{A}_t=(\tilde{X}_t^y)^2+ \tilde{X}_t^yX_t^y + (X_t^y)^2$, we have
		\begin{align*}
			\dif \tilde{Z}_t=(a-b\tilde{A}_t)\tilde{Z}_t\dif t + \lambda Z_t \dif t+[\sigma(\tilde{X}_{t-}^y)-\sigma(X_{t-}^y)] \ \dif N_t, \ \ \ \tilde{Z}_0=0.
		\end{align*}
		It is easy to check that
		\begin{align*}
			\tilde{Z}_t=\lambda e^{\int_0^t (a-b\tilde{A_s})\dif s}\int_0^t e^{\int_0^s (b\tilde{A}_r-a) \dif r} Z_s \dif s+ e^{\int_0^t (a-b\tilde{A_s})\dif s}\int_0^t[ \sigma(\tilde{X}_{s-}^y)-\sigma(X_{s-}^y)] \ \dif N_s.
		\end{align*}
		Let $y\geq\sqrt\frac ab, x\geq \sqrt\frac{a}{2b}$, we have
		\begin{align*}
			X_t^x\geq\sqrt\frac{a}{2b},\ \ \  X_t^y\geq\sqrt\frac ab,\ \ \ \tilde{X}_t^y\geq \tilde{Y}_t^y\geq p\land \sqrt\frac ab.
		\end{align*}
		where $p$ is a stable equilibrium point of the following equation
		\begin{align*}
			\dif \tilde{Y}^y_t=[(a-\lambda)\tilde{Y}^y_t-b\cdot (\tilde{Y}_t^y)^3+(a+L_{\sigma}+\frac{L_{\sigma}^2}2)\sqrt\frac{a}{2b}] \dif t,\ \ \  \tilde{Y}^y_0=y.
		\end{align*}
		This implies that $\tilde{A}_t\geq\frac ab + q^2$, $q=p\land\sqrt\frac ab $. Then, by \eqref{ineq1}, $\tilde{A}_t\geq\frac ab + q^2$ and  $0<\sigma\le M$,
		\begin{align}\label{ineq2}
			\EX|\tilde{Z}_t|\le& \lambda\EX| e^{\int_0^t (a-b\tilde{A_s})\dif s}\int_0^t e^{\int_0^s (b\tilde{A}_r-a) \dif r} Z_s \dif s|+\EX\bigg[e^{\int_0^t (a-b\tilde{A_s})\dif s}\int_0^t  |\sigma(\tilde{X}_{s-}^y)-\sigma(X_{s-}^y)| \ \dif N_s\bigg]\nonumber\\
			\le & \frac{2\lambda|x-y|(1-e^\frac{(2a-2\lambda+2L_{\sigma}+L_{\sigma}^2)t}{2})}{2\lambda-2a-2L_{\sigma}-L_{\sigma}^2} +2Mte^{-bq^2t} \rightarrow 0, \ \ \ \text{as} \ \ \ t\rightarrow \infty, \ \ \ x\rightarrow y.
		\end{align}
		Then, if $\lambda> (a+L_{\sigma})+\frac{L_{\sigma}^2}2$, $x\geq \sqrt\frac{a}{2b} $ and $ y\geq\sqrt\frac ab$, by \eqref{ineq1} and \eqref{ineq2} for any $f\in L_b(\mathbb{R})$ we can obtain,
		\begin{align*}
			|\EX f(X_t^x)-\EX f(X_t^y)|\le& \EX|f(X_t^x)- f(\tilde{X}_t^y)| + \EX |f(X_t^y)- f(\tilde{X}_t^y)|\\
			\le& L_f(\EX|Z_t|+\EX |\tilde{Z}_t|)\rightarrow 0,
		\end{align*}
		as $t\rightarrow\infty$, and $x\rightarrow y$, where $L_f$ is Lipschitz constant of $f$. This implies that $\{P_t\}_{t\geq 0}$ is eventually continuous at $y$. However, it is not straightforward to demonstrate that $\{P_t\}_{t\geq 0}$ is eventually continuous on $(-\infty, \sqrt\frac ab)$ using the methods previously discussed.
		
		\textbf{Step 2.} Fixed $y=  \sqrt\frac ab $ we prove that for any $x\in \mathbb{R}$ there exist $\gamma, C>0$ such that
		\begin{align}\label{estipossion}
			\EX|X_t^x-\sqrt\frac ab|^2\le C|x-\sqrt\frac ab|^2e^{-\gamma t}+C.
		\end{align}
		By It\^{o} formula we have,
		\begin{align*}
			\dif |X_t-\sqrt\frac ab|^2=2(X_t-\sqrt\frac ab)(aX_t-bX_t^3)\dif t +\big[2(X_t-\sqrt\frac ab)\sigma(X_{s-})+|\sigma(X_{s-})|^2\big] \dif N_t.
		\end{align*}
		Now integrate with respect to $t$ and take expectations to obtain,
		\begin{align*}
			\EX|X_t-\sqrt\frac ab|^2=&|x-\sqrt\frac ab|^2+ \int_0^t\EX\bigg[ 2(X_s-\sqrt\frac ab)(aX_s-bX_s^3)+2(X_s-\sqrt\frac ab)\sigma(X_s)+|\sigma(X_s)|^2 \bigg]\dif s\\
			\le&|x-\sqrt\frac ab|^2+\int_0^t\EX\bigg[-2bX_s^4+2\sqrt{ab}X_s^3+(2a+1)X_s^2-2\sqrt\frac {a^3}bX_s+3M^2+ \frac ab \bigg]\dif s\\
			\le&|x-\sqrt\frac ab|^2+\int_0^t\EX\bigg[-bX_s^4+(3a+2)X_s^2+3M^2+ \frac ab+\frac {a^3}b \bigg]\dif s.
		\end{align*}
		It is easy to check that
		\begin{align*}
			&\EX|X_t-\sqrt\frac ab|^2+\int_0^t \EX|X_s-\sqrt\frac ab|^2 \dif s\\
			\le&|x-\sqrt\frac ab|^2+\int_0^t\EX\bigg[-bX_s^4+(3a+4)X_s^2+3M^2+\frac{3a}b+\frac {a^3}b \bigg] \dif s\\
			\le &|x-\sqrt\frac ab|^2+\bigg[\frac{(3a+4)^2}{4b}+3M^2+\frac{3a}b+\frac {a^3}b\bigg]t.
		\end{align*}
		Using Gronwall inequality, there exists $C=C(a,b,M)>0$ such that
		\begin{align*}
			\EX|X_t-\sqrt\frac ab|^2\le C|x-\sqrt\frac ab|^2e^{-t}+C.
		\end{align*}
		Additionally, by using Markov-Chebyshev inequality, for any $r>0$
		\begin{align*}
			P_t(x, B(\sqrt\frac ab,r))=\EX\mathbf{1}_{|X_t^x-\sqrt\frac ab|<r}=1-\PX(|X_t^x-\sqrt\frac ab|\geq r)&\geq 1-\dfrac{\EX[|X_t^x-\sqrt\frac ab|]^2}{r^2}\\
			&\geq 1-\dfrac{Ce^{-t}|x-\sqrt\frac ab|^2}{r^2}-\dfrac{C}{r^2}.
		\end{align*}
		Then, there exists sufficiently $r_0>0$  such that
		\begin{align}\label{jg1}
			\liminf_{t\rightarrow \infty} P_t(x, B(\sqrt\frac ab,r))\geq \frac12, \ \ \ \forall r>r_0.
		\end{align}
		
		\textbf{Step 3.} We prove for any $\epsilon>0$ and sufficiently large $r>0$ (the lower bound for the value of $r$ will be given in the proof below) there exists $T>0$ such that
		\begin{align*}
			\inf_{x\in B(\sqrt\frac ab,r)} P_T(x, B(\sqrt\frac ab,\epsilon))>0.
		\end{align*}
		
		We define
		\begin{align*}
			\tau_0=0, \ \ \ \text{and}\ \ \ \tau_n:=\inf\{t\geq 0; N_t-N_{\tau_{n-1}}\neq 0 \},\ \ \ \forall n\geq 1.
		\end{align*}
		Consider the following equation:
		\begin{align}
			\dif Y_t= (aY_t- b Y_t^3) \dif t, \ \ \ Y_0=x\in\mathbb{R}.
		\end{align}
		Take $Y_t^x$ solving this equation with initial datum $x$.
		Let $\tilde{\delta}>0$ satisfying $0<\tilde{\delta}<\min\{ \tfrac m3, |\sqrt\frac ab-m|, \sqrt\frac ab\}$. We prove that there exists $T>0$ such  that the infimum of $P_{T}(x,B(\sqrt\frac ab,\epsilon))$  is positive for $x$ in each of the intervals $[\tilde{\delta}, \sqrt\frac ab+r], [-\tilde{\delta}, \tilde{\delta}] $ and $ [\sqrt\frac ab-r, -\tilde\delta ]$. The proof is divided into the following three cases:\\
		\textbf{Case 1:} $\inf\limits_{x\in [\tilde{\delta}, \sqrt\frac ab+r]} P_{T}(x,B(\sqrt\frac ab,\epsilon))>0$\\
		For any $\epsilon>0$ and sufficiently large $r>0$ there exists $T_1>0$ such that
		\begin{align}\label{in1}
			|Y_t^x-\sqrt\frac ab|<\epsilon,\ \ \ \forall x\in[\tilde{\delta}, \sqrt\frac ab+r],\ \ \ \forall t\geq T_1.
		\end{align}
		Since $\tau_1$ has exponential distribution with mean $1$, for any $t\geq0$,
		\begin{align*}
			\PX(\tau_1>t)=e^{-t}>0.
		\end{align*}
		Then, by \eqref{in1}, for such $r>0$, $\epsilon>0$ and $t\geq T_1$
		\begin{align}\label{irreeq1}
			\inf_{x\in [\tilde{\delta}, \sqrt\frac ab+r]}\PX(|X_{t}^x-\sqrt\frac ab|<\epsilon)
			\geq\PX(\tau_1>{t})=e^{-{t}} .
		\end{align}
		This implies that
		\begin{align}\label{jieguo1}
			\inf_{x\in [\tilde{\delta}, \sqrt\frac ab+r]} P_{t}(x,B(\sqrt\frac ab,\epsilon))>0, \ \ \ \forall t\geq T_1.
		\end{align}
		\textbf{Case 2.} $\inf\limits_{x\in [-\tilde{\delta}, \tilde{\delta}]} P_{T}(x,B(\sqrt\frac ab,\epsilon))>0$\\
		Since $Y_t^x$ is continuous with respect to the initial value $x$, there exists $t_{\tilde{\delta}}>0$ such that
		\begin{align*}
			|Y^x_{t_{\tilde{\delta}}}|<2\tilde{\delta},\ \ \  \forall x\in[-\tilde{\delta}, \tilde{\delta}].
		\end{align*}
		Let $\tau_1<t_{\tilde{\delta}}$, then we have for any $x\in [-\tilde{\delta}, \tilde{\delta}]$,
		\begin{align*}
			X_{\tau_1}^x=Y_{\tau_1}^x+\sigma(Y_{\tau_1}^x)\in B(\zeta_1, 2\tilde\delta),
		\end{align*}
		where $m<\zeta_1<M$, and $\tilde\delta<\frac m3$. It follows that $\tilde\delta<X_{\tau_1}^x<M+2\tilde\delta$. Without loss of generality, assume that $r>(M+2\tilde\delta)\lor r_0$. Then, we obtain that
		\begin{align*}
			|   X_t^x-\sqrt\frac ab |<\epsilon, \ \ \ \text{when}\ \ \ \tau_2>t>T_1+t_{\tilde\delta}, \ \ \ \tau_1<t_{\tilde\delta}.
		\end{align*}
		This implies that
		\begin{align}\label{jieguo2}
			\inf_{x\in [-\tilde{\delta}, \tilde{\delta}]} P_{t}(x,B(\sqrt\frac ab,\epsilon))\geq& \PX(\tau_2>t>T_1+t_{\tilde\delta},  \tau_1<t_{\tilde\delta})
			\geq\PX(\tau_2-\tau_1>t, \tau_1<t_{\tilde\delta})\nonumber\\
			=&\PX(\tau_2-\tau_1>t)\PX( \tau_1<t_{\tilde\delta})
			=(e^{-t})(1-e^{-t_{\tilde\delta} }), \ \ \ \forall t\geq T_1+t_{\tilde\delta}.
		\end{align}
		\textbf{Case 3.} $\inf\limits_{x\in[\sqrt\frac ab-r,-\tilde{\delta} ] }P_{T}(x, B(\sqrt\frac ab,\epsilon))>0$\\
		Let $n=[\frac{\sqrt\frac{4a}b}{m}]+1$, $0<\delta_0<\frac{\sqrt\frac ab -\tilde\delta}{n+1}$, there exists $T_2>0$ such that
		\begin{align*}
			|Y_t^x+\sqrt\frac ab|<\delta_0, \ \ \ \forall x\in[\sqrt\frac ab-r, -\tilde\delta ], \ \ \ \forall t\geq T_2.
		\end{align*}
		Let $\tau_1>T_2$, we have
		\begin{align*}
			-\sqrt\frac ab+m-\delta_0<X_{\tau_1}^x=Y_{\tau_1}^x+\sigma(Y_{\tau_1}^x)<-\sqrt\frac ab+M+\delta_0.
		\end{align*}
		There exists $t_1>0$ such that
		\begin{align*}
			-\sqrt\frac ab+m-2\delta_0< Y_{t_1+\tau_1}^{X_{\tau_1}^x}<-\sqrt\frac ab+M+2\delta_0
		\end{align*}
		Then, let $\tau_2-\tau_1<t_1$, we have
		\begin{align*}
			-\sqrt\frac ab+2m-2\delta_0<X_{\tau_2}^x=Y_{\tau_2}^{X_{\tau_1}^x}+\sigma(Y_{\tau_2}^{X_{\tau_1}^x})<-\sqrt\frac ab+2M+2\delta_0.
		\end{align*}
		There exists $t_2>0$ such that
		\begin{align*}
			-\sqrt\frac ab+2m-3\delta_0<Y_{t_2+\tau_2}^{X_{\tau_2}^x}<-\sqrt\frac ab+2M+3\delta_0.
		\end{align*}
		By induction, we can obtain there exists $t_{n-1}>0$ such that
		\begin{align*}
			-\sqrt\frac ab+(n-1)m-n\delta_0<    Y_{t_{n-1}+\tau_{n-1}}^{X_{\tau_{n-1}}^x}<-\sqrt\frac ab+(n-1)M+n\delta_0 .
		\end{align*}
		Then, let $\tau_n-\tau_{n-1}<t_{n-1}$, we have
		\begin{align*}
			-\sqrt\frac ab+nm-n\delta_0<X_{\tau_n}^x=Y_{\tau_n}^{X_{\tau_{n-1}}^x}+\sigma(Y_{\tau_{n}}^{X_{\tau_{n-1}}^x})<-\sqrt\frac ab+nM+n\delta_0.
		\end{align*}
		Without loss of generality, assume that $r>(-\sqrt\frac ab+nM+n\delta_0)\lor r_0$. This implies that
		\begin{align*}
			\tilde\delta<\xi_n-n\delta_0<   X_{\tau_n}^x< -\sqrt\frac ab+nM+n\delta_0<r.
		\end{align*}
		Then, let $\tau_{n+1}>t\geq T_1+T_2+2t_1+\cdots+t_{n-1}$, and $T_2<\tau_1<T_2+t_1$ we have
		\begin{align*}
			|X_{t}^x-\sqrt\frac ab| <\epsilon, \ \ \ \forall x\in[\sqrt\frac ab-r, -\tilde\delta ].
		\end{align*}
		It follows that $\forall t\geq T_1+T_2+2t_1+\cdots+t_{n-1}$,
		\begin{align}\label{jieguo3}
			&\inf_{x\in[\sqrt\frac ab-r, -\tilde\delta ]}P_{t}(x, B(\sqrt\frac ab,\epsilon))\nonumber\\
			\geq& \PX(T_2<\tau_1<T_2+t_1,\tau_2-\tau_1<t_1,\cdots, \tau_n-\tau_{n-1}<t_{n-1}, \tau_{n+1}>t)\nonumber\\
			\geq&\PX(T_2<\tau_1<T_2+t_1)\PX(\tau_2-\tau_1<t_1)\cdots\PX(\tau_n-\tau_{n-1}<t_{n-1})\times \nonumber\\
			&\PX(\tau_{n+1}-\tau_n>t+T_1+T_2+2t_1+t_2+\cdots+t_{n-1})\nonumber\\
			=&(e^{-T_2}-e^{-T_2-t_1})(1-e^{-t_1})\cdots(1-e^{-t_n})\times
			(e^{-t-T_1-T_2-2t_1-t_2-\cdots-t_{n-1}}).
		\end{align}
		By \eqref{jieguo1}, \eqref{jieguo2} and\eqref{jieguo3}, let $T=T_1+t_{\tilde\delta}+T_2+2t_1+t_2+\cdots+t_{n-1}$, then we have
		\begin{align*}
			\inf_{x\in[\sqrt\frac ab-r,\sqrt\frac ab+r ] }P_{T}(x, B(\sqrt\frac ab,\epsilon))>0,
		\end{align*}
		where $r>(M+2\tilde\delta)\lor(-\sqrt\frac ab+nM+n\delta_0)\lor r_0$.
		Since $\{P_t\}_{t\geq 0}$ is eventually continuous at $\sqrt\frac ab$ and \eqref{estipossion} holds, we can conclude from Corollary \ref{col1} that $\{P_t\}_{t\geq 0}$ is asymptotically stable.\bigskip\qed\end{proof}
	
	\begin{remark}
		 Compared to the example in \cite[Section 5.1]{Gong3} or Example~\eqref{langevin} which is not asymptotically stable, our example \eqref{pola}, driven by a non-degenerate and non-Gaussian noise, is asymptotically stable. We use eventual continuity at a single point to establish its asymptotic stability. This example can be shown to be globally eventually continuous based on the equivalent conditions in \cite[Theorem 3.1]{Gong2}. However, it appears challenging to prove global eventual continuity directly.  
	\end{remark}

	\section{Proofs}\label{proofs}
	In this section, we provide the proofs corresponding to the theorems and corollary.
	
	\subsection{Proof of Theorem \ref{thm0}}
	We observe that the semigroup satisfies eventual continuity for all bounded measurable functions, and therefore Theorem \ref{thm0} holds. Although the proof follows arguments similar to those in \cite[Theorem 3.16]{Gong3} and \cite[Theorem 3.3]{Feller}, we provide a proof here for completeness.
	
	\emph{Proof of Theorem \ref{thm0}.} $(1) \Rightarrow (2):$ Let $\mu$ be a unique invariant measure for $\{P_t\}_{t\geq 0}$, and $z\in \text{supp}\ \mu$, for all $x\in E$
	\begin{align*}
		\liminf_{t\rightarrow\infty} P_t(x, B(z,\epsilon))\geq\mu(B(z,\epsilon))>0
	\end{align*}
	and for every $f\in B_b(E)$,
	\begin{align*}
		&\limsup_{x\rightarrow z}\limsup_{t\rightarrow \infty}\sup_{||f||_{\infty}=1}|P_tf(x)-P_tf(z)|\\
		\le& \limsup_{x\rightarrow z}\limsup_{t\rightarrow \infty}\sup_{||f||_{\infty}=1}\bigg[|\int_E P_tf(y) \delta_x(\dif y)-\langle \mu, f \rangle|+|\int_E P_tf(y) \delta_z(\dif y)-\langle \mu, f \rangle|\bigg]\\
		=&\limsup_{x\rightarrow z}\limsup_{t\rightarrow \infty}\bigg(||P_t^*\delta_x-\mu||_{TV}+||P_t^*\delta_y-\mu||_{TV}\bigg)=0.
	\end{align*}

	The proof of $(2) \Rightarrow (1)$ is divided into three steps. 
	
	\textbf{Step 1.} Due to \cite[Theorem 3.16]{Gong3}, we know that $\{P_t(z, \cdot)\}_{t\geq 0}$ is tight. Consequently, there exists an invariant measure.

	\textbf{Step 2.} We are going to show that $\lim\limits_{t\rightarrow\infty}\sup\limits_{||f||_{\infty}= 1}|P_tf(x_1)-P_tf(x_2)|=0$ for all $x_1,x_2\in E$. As $\{P_t\}_{t\geq 0}$ is $TV$-eventually continuous at $z$, there is $\delta=\delta(\epsilon)>0$ such that
	\begin{align*}
		\limsup_{t\rightarrow\infty}\sup_{||f||_{\infty}= 1}|P_tf(x)-P_tf(z)|<\frac\epsilon4,\ \ \forall x\in B(z,\delta).
	\end{align*}
	For such $\delta>0$, \eqref{C4} gives some $\alpha\in(0,\frac12)$ such that $\liminf\limits_{t\rightarrow\infty}P_t(x, B(z,\frac\delta2))>\alpha$ for all $x\in E$. Then by Fatou's lemma, for any $\mu\in \mathcal{M}(E)$ we have
	\begin{align}\label{1-1}
		\liminf_{t\rightarrow\infty}P_t^*\mu(B(z,\frac\delta2))\geq&\int_E\liminf_{t\rightarrow\infty}P_t(y, B(z,\frac\delta2))\mu(\dif y)>\alpha.
	\end{align}
	Fixed $x_1, x_2\in E$, by induction we can define four sequences of probability measures $\{\nu_i^{x_1}\}_{i\geq1}$, $\{\mu_i^{x_1}\}_{i\geq1}$, $\{\nu_i^{x_2}\}_{i\geq1}$, $\{\mu_i^{x_2}\}_{i\geq1} $ and a sequence of positive numbers $\{t_i\}_{i\geq1}$ such that ${\text{supp}}\ \nu_i^{x_j}\subset B(z,\delta)$, ${\text{supp}}\ \mu_i^{x_j}\subset E$ and
	\begin{align*}
		P_{t_{i}}^*\mu_{i-1}^{x_j}=\alpha\nu_{i}^{x_j}+(1-\alpha)\mu_{i}^{x_j}, \ \ j=1,2 \ \ {\text{and}} \ \ i\geq 1.
	\end{align*}
	Indeed, set $\mu_0^{x_j}=\delta_{x_j}$,\ $j=1,2$, and $t_0=0$. By \eqref{1-1}, we choose $t_1>t_0$ be such that
	\begin{align*}
		P_{t}^*\mu_{0}^{x_j}(B(z,\frac\delta2))>\alpha, \ \ \forall j=1,2 \ \ \ \text{and}\ \ \ t\geq t_1.
	\end{align*}
	Let
	\begin{align*}
		\nu_{1}^{x_j}(\cdot)=\dfrac{P_{ t_{1}}^*\mu_{0}^{x_j}(\cdot\cap B(z,\frac\delta2))}{P_{ t_{1}}^*\mu_{0}^{x_j}(B(z,\frac\delta2))},\ \ \mu_{1}^{x_j}(\cdot)=\frac{1}{1-\alpha}(P_{ t_{1}}^*\mu_{0}^{x_j}(\cdot)-\alpha\nu_{1}^{x_j}(\cdot)),\  j=1,2.
	\end{align*}
	Assume that we have done it for $i=1,\cdots, k-1$. Now by \eqref{1-1}, let $t_k> t_{k-1}$ be such that
	\begin{align*}
		P_{t}^*\mu_{k}^{x_j}(B(z,\frac\delta2))>\alpha, \ \ \forall j=1,2 \ \ {\text{and}} \ \ t\geq t_k.
	\end{align*}
	Let
	\begin{align*}
		\nu_{k}^{x_j}(\cdot)=\dfrac{P_{ t_{k}}^*\mu_{k-1}^{x_j}(\cdot\cap B(z,\frac\delta2))}{P_{ t_{k}}^*\mu_{k-1}^{x_j}(B(z,\frac\delta2))},\ \ \mu_{k}^{x_j}(\cdot)=\frac{1}{1-\alpha}(P_{ t_{k}}^*\mu_{k-1}^{x_j}(\cdot)-\alpha\nu_{k}^{x_j}(\cdot)),\  j=1,2.
	\end{align*}
	Then it follows that
	\begin{align*}
		P_{t_1+\cdots+t_k}^*\delta_{x_j}(\cdot)=&\alpha P_{t_2+\cdots+t_k}^*\nu_1^{x_j}(\cdot)+\alpha(1-\alpha)P_{t_3+\cdots+t_k}^*\nu_2^{x_j}(\cdot)+\cdots\\
		&+\alpha(1-\alpha)^{k-2}P_{t_k}^*\nu_{k-1}^{x_j}(\cdot)+\alpha(1-\alpha)^{k-1}\nu_k^{x_j}(\cdot)+(1-\alpha)^k\mu_k^{x_j}(\cdot),
	\end{align*}
	where $\text{supp}\ \nu_i^{x_j}\subset \overline{B(z,\frac\delta2)}\subset B(z,\delta), i=1,\cdots ,k, j=1,2$. Thus from Fatou's lemma and $TV$-eventually continuous at $z$, we have
	\begin{align*}
		&\limsup_{t\rightarrow\infty}\sup_{||f||_{\infty}= 1}|\langle P_{t}f, \nu_i^{x_1}\rangle-\langle P_{t}f, \nu_i^{x_2}\rangle\\
		\le&\limsup_{t\rightarrow\infty}\sup_{||f||_{\infty}= 1}\bigg[|\langle P_{t}f-P_{t}f(z), \nu_i^{x_1}\rangle|+|\langle P_{t}f-P_{t}f(z), \nu_i^{x_2}\rangle |\bigg]\le \frac\epsilon2.
	\end{align*}
	Finally, using the measure decomposition gives
	\begin{align*}
		& \limsup_{t\rightarrow \infty}\sup_{||f||_{\infty}= 1}|P_{t}f(x_1)-P_{t}f(x_2)|\\
		=&\limsup_{t\rightarrow \infty}\sup_{||f||_{\infty}= 1}|\langle P_{t}f, P_{t_1+\cdots+t_k}^*\delta_{x_1}\rangle- \langle P_{t}f, P_{t_1+\cdots+t_k}^*\delta_{x_2}\rangle|\\
		\le& \limsup_{t\rightarrow \infty}\sup_{||f||_{\infty}= 1}\bigg(\alpha|\langle P_{t}f, \nu_1^{x_1} \rangle- \langle P_{t}f, \nu_1^{x_2} \rangle|+\alpha(1-\alpha)|\langle P_{t}f, \nu_2^{x_1} \rangle- \langle P_{t}f, \nu_2^{x_2} \rangle |\\
		&+\cdots+\alpha(1-\alpha)^{k-1}|\langle P_{t}f,\nu_k^{x_1} \rangle- \langle P_{t}f, \nu_k^{x_2} \rangle |+(1-\alpha)^k|\langle P_{t}f,\mu_k^{x_1} \rangle- \langle P_{t}f, \mu_k^{x_2} \rangle|\bigg)\\
		\le&(\alpha+\cdots+\alpha(1-\alpha)^{k-1})\frac\epsilon2+2(1-\alpha)^k<\epsilon,
	\end{align*}
	in the last inequality we can choose sufficiently large $k\in \mathbb{N}$ such that
	\begin{align*}
		2(1-\alpha)^k<\frac\epsilon2.
	\end{align*}
	 Since $\epsilon$ was arbitrary, we obtain
	 \begin{align*}
	 \lim_{t\rightarrow\infty}\sup_{||f||_{\infty}= 1}|P_tf(x_1)-P_tf(x_2)|=0.	
	 \end{align*}

	\textbf{Step 3.}   Let $\mu$ is an invariant measure for $\{P_t\}_{t\geq 0}$.   For any $\nu\in\mathcal{M}(E)$, and $x_2\in E$
	\begin{align*}
		&| \int_{E}f(x_1)P_t^*\nu(\dif x_1)-\int_E f(x_1) \mu(\dif x_1)|=| \int_{E}P_tf(x_1)\nu(\dif x_1)-\int_E P_tf(x_1) \mu(\dif x_1)|\\
		\le&  \int_{E}|P_tf(x_1)-P_tf(x_2)|\nu(\dif x_1)+\int_E |P_tf(x_1)- P_tf(x_2)| \mu(\dif x_1),
	\end{align*}
	using the dominated convergence theorem and the face in \textbf{Step 2} we obtain
	\begin{align*}
		\lim_{t\rightarrow \infty}\sup_{||f||_{\infty}= 1}\bigg(\int_{E}|P_tf(x_1)-P_tf(x_2)|\nu(\dif x_1)+\int_E |P_tf(x_1)- P_tf(x_2)|\bigg)=0.
	\end{align*}
	This implies \eqref{TV}. The proof is complete.\qed\bigskip

	\subsection{Proof of Theorem~\ref{thm2}}
	\emph{Proof of Theorem \ref{thm2}.} $(1) \Rightarrow (2):$ Let $\mu$ be a unique invariant measure for $\{P_t\}_{t\geq 0}$, and $z\in \text{supp}\ \mu$, since $\{P_t\}_{t\geq 0}$ is asymptotically stable, from \cite[Theorem 2.1]{Bill},  then for all $x\in E$ and $\epsilon>0$
	\begin{align*}
		\liminf_{t\rightarrow\infty} P_t(z, B(z,\epsilon))\geq\mu(B(z,\epsilon))>0.
	\end{align*}
	Since $\mathcal{T}=E$, it follows that for all $x\in E$ and $\epsilon>0$
	\begin{align*}
		\limsup_{t\rightarrow\infty} Q_t(x, B(z,\epsilon))\geq\mu(B(z,\epsilon))>0,
	\end{align*}
	and for every $f\in L_b(E)$, $y\in E$
	\begin{align*}
		&\limsup_{x\rightarrow y}\limsup_{t\rightarrow \infty}|P_tf(x)-P_tf(y)|\\
		\le& \limsup_{x\rightarrow y}\limsup_{t\rightarrow \infty}\bigg[|\int_E P_tf(\xi) \delta_x(\dif \xi)-\langle \mu, f \rangle|+|\int_E P_tf(\xi) \delta_y(\dif \xi)-\langle \mu, f \rangle|\bigg]=0.
	\end{align*}
	
	$(2)\Rightarrow (1):$ Firstly, since $\mathcal{T}=E$, we can infer the existence of an invariant measure for $\{P_t\}_{t\geq 0}$. For the remainder of the proof, we divide it into three steps.

	\textbf{Step 1.} We first show that fix $0<\epsilon<1$, we have
	\begin{align}\label{QC0}
		\liminf_{t\rightarrow\infty} Q_t(x, B(z,\epsilon))>0, \ \ \forall x\in E.
	\end{align}

	From \cite[Corollary 4.8]{decom}, we have an ergodic measure $\mu_*$ for $\{P_t\}_{t\geq 0}$. By using \cite[Theorem 4.4]{decom} (or \cite[Theorem 4.2]{decom2} ), we have $Q_t^*\delta_z$ converges weakly to $\mu_*$ as $t\rightarrow \infty$. And from \cite[Theorem 2.1]{Bill},
	\begin{align*}
		0<\limsup_{t\rightarrow \infty}Q_t(z, \overline{B(z,\frac\epsilon4)})\le\mu_*
		(\overline{B(z,\frac\epsilon4)})\le \mu_*(B(z,\frac\epsilon2)).
	\end{align*}
	It is easy to obtain
	\begin{align*}
		\alpha:=\liminf_{t\rightarrow \infty}Q_t(z,B(z,\frac\epsilon2))\geq \mu_*(B(z,\frac\epsilon2))>0.
	\end{align*}
	Let $f(x)=(1-\frac1\epsilon d(x,z))\lor 0$, then $||f||_{\infty}=1$, and  $ L_f \le \frac1\epsilon$, where $L_f$ is Lipschitz constant of $f$. Moreover, $\frac12 \mathbf{1}_{B(z,\frac\epsilon2)}\le f\le\mathbf{1}_{B(z,\epsilon)}$.
	Since $\{P_t\}_{t\geq 0}$ is eventually continuous at $z$, it follows that $\{Q_t\}_{t\geq 0}$ is eventually continuous at $z$ (see \cite{Gong1}). Then, we may choose $\delta>0$, for any $y\in B(z,\delta)$,
	\begin{align*}
		\limsup_{t\rightarrow \infty}|Q_tf(y)-Q_tf(z)|<\frac\alpha4.
	\end{align*}
	Then, for any $y\in B(z,\delta)$
	\begin{align}\label{Q0}
		\liminf_{t\rightarrow\infty}Q_t(y, B(z,\epsilon))>\frac12\liminf_{t\rightarrow\infty}Q_t(z, B(z,\frac\epsilon2))-\limsup_{t\rightarrow \infty}|Q_tf(y)-Q_tf(z)|\geq \frac\alpha4.
	\end{align}
	For any $x\in E$, from \eqref{C1} we know that there exists a sufficiently large $t_0>0$ such that
	\begin{align*}
		Q_{t_0}(x, B(z,\delta))>0.
	\end{align*}
	Then, by using Fatou's lemma for any $x\in E$
	\begin{align*}
		\liminf_{t\rightarrow\infty}Q_t(x, B(z,\epsilon))=\liminf_{t\rightarrow\infty}Q_{t+t_0}(x, B(z,\epsilon))\geq\int_{B(z,\delta)}\liminf_{t\rightarrow\infty}Q_t(y,B(z,\epsilon)) Q_{t_0}(x,\dif y)>0.
	\end{align*}
	
	\textbf{Step 2.} Now, we show for any $\nu\in \mathcal{M}(E)$, we have $Q_t^*\nu$ converges weakly to $\mu_*$, and $\mu_*$ is a unique invariant measure for $\{P_t\}_{t\geq 0}$. This means that $\{P_t\}_{t\geq 0}$ is weak-* mean ergodic(see \cite[Denfinition 2.2]{epro}).     We only need to prove for any $x,y\in E$ and $f\in L_b(E)$,
	\begin{align}\label{s20}
		\limsup_{t\rightarrow \infty}|Q_tf(x)-Q_tf(y)|=0.
	\end{align}
	
	Set $\{\gamma_n\}_{n\in\mathbb{N}^+}$ be a sequence of positive numbers monotonically decreasing to $0$. And fix arbitrary $\epsilon>0, f \in L_{b}(E), x_{1}, x_{2} \in E$. We define $\mathcal{D} \subset$ $\mathbb{R}$ in the following way: $\lambda \in \mathcal{D}$ if and only if $\lambda>0$ and there exists a positive integer $N$, a sequence of times $\{T_{\lambda, n}\}_{n\in\mathbb{N}^+}$ and sequences of measures $\{\mu_{\lambda, i}^{n}\}_{n\in\mathbb{N}^+},\{\nu_{\lambda, i}^{n}\}_{n\in\mathbb{N}^+} \subset$ $\mathcal{M}(E), i=1,2$, such that for $n \geq N$,
	\begin{align}
		T_{\lambda, n} & \geq n , \label{s21} \\
		\left\|Q_{T_{\lambda, n}}\left(x_{i}, \cdot\right)-\mu_{\lambda, i}^{n}\right\|_{\mathrm{TV}} & <\gamma_n , \label{s22} \\
		\mu_{\lambda, i}^{n} & \geq \lambda v_{\lambda, i}^{n} \quad \text { for } i=1,2, \label{s23}
	\end{align}
	and
	\begin{equation}
		\limsup _{T \rightarrow \infty}\left|\int_{E} f(x) Q_T^* v_{\lambda, 1}^{n}(\dif  x)-\int_{E} f(x) Q_T^* v_{\lambda, 2}^{n}(\dif  x)\right|<\epsilon \label{s24}.
	\end{equation}
	Let us state the following claim whose proof is postponed to Section \ref{appendix}.\\
	\textbf{Claim:}:    For given $\epsilon>0,\{\gamma_n\}_{n\in\mathbb{N}^+}, x_{1}, x_{2} \in E$ and $f \in L_b(E)$, the set $\mathcal{D} \neq \varnothing$. Moreover, we have $\sup \mathcal{D}=1$.

	Admitting this claim, it follows that for any $\epsilon>0$ there exists an $\lambda>1-\epsilon$ that belongs to $\mathcal{D}$, such that for any $T\geq T_{\lambda,n}$
	\begin{align}\label{s25}
		& \mid \int_{E} f(y) Q_T\left(x_{1}, \dif  y\right)-\int_{E} f(y) Q_T\left(x_{2}, \dif  y\right) \mid \nonumber \\
		\leq    & \sum_{i=1}^{2}\left|\int_{E} f(y) Q_T\left(x_{i}, \dif  y\right)-\int_{E} f(y) Q_{T, T_{\lambda, n}}\left(x_{i}, \dif  y\right)\right| \nonumber \\
		&+\left|\int_{E} f(y) Q_T^* \mu_{\lambda, 1}^{n}(\dif  y)-\int_{E} f(y) Q_T^* \mu_{\lambda, 2}^{n}(\dif  y)\right|\nonumber \\
		&+\sum_{i=1}^{2}\left|\int_{E} f(y) Q_{T, T_{\lambda, n}}\left(x_{i}, \dif  y\right)-\int_{E} f(y) Q_T^* \mu_{\lambda, i}^{n}(\dif  y)\right| \nonumber \\
		\leq &\sum_{i=1}^{2}\left|\int_{E} f(y) Q_T\left(x_{i}, \dif  y\right)-\int_{E} f(y) Q_{T, T_{\lambda, n}}\left(x_{i}, \dif  y\right)\right| \nonumber \\
		& +\left|\int_{E} f(y) Q_T^* \mu_{\lambda, 1}^{n}(\dif  y)-\int_{E} f(y) Q_T^* \mu_{\lambda, 2}^{n}(\dif  y)\right|
		+2 \gamma_n\|f\|_{\infty}.
	\end{align}
	We can use condition \eqref{s23} to handle the second term on the last right-hand side of \eqref{s25}, and then we can replace $\mu_{\lambda, i}^{n}$ by $v_{\lambda, i}^{n}$ and obtain
	\begin{align}\label{s26}
		& \left|\int_{E} f(y) Q_T^* \mu_{\lambda, 1}^{n}(\dif  y)-\int_{E} f(y) Q_T^* \mu_{\lambda, 2}^{n}(\dif  y)\right| \nonumber\\
		{\leq}& \lambda\left|\int_{E} f(y) Q_T^* v_{\lambda, 1}^{n}(\dif  y)-\int_{E} f(y) Q_T^* v_{\lambda, 2}^{n}(\dif  y)\right| +\sum_{i=1}^{2}\|f\|_{\infty}\left(\mu_{\lambda, i}^{n}-\lambda v_{\lambda, i}^{n}\right)(E) \nonumber\\
		\leq&\left|\int_{E} f(y) Q_T^* v_{\lambda, 1}^{n}(\dif  y)-\int_{E} f(y) Q_T^* v_{\lambda, 2}^{n}(\dif  y)\right|+2 \epsilon\|f\|_{\infty}.
	\end{align}
	In the last inequality, we have used the fact that $1-\lambda<\epsilon$. Summarizing, from \cite[Lemma 2]{epro}, \eqref{s25}, \eqref{s26} and \eqref{s24}, we obtain that
	\begin{align*}
		\underset{T \rightarrow \infty}{\limsup }\left|\int_{E} f(y) Q_T\left(x_{1}, \dif  y\right)-\int_{E} f(y) Q_T\left(x_{2}, \dif  y\right)\right| \leq 2 \gamma_n\|f\|_{\infty}+2 \epsilon\|f\|_{\infty}+\epsilon.
	\end{align*}
	Since $\epsilon>0$ and $n$ were arbitrarily chosen, we conclude that \eqref{s20} follows.

	\textbf{Step 3.} We are going to show that, for every $\epsilon>0$
	\begin{align}\label{C14}
		\inf_{x\in E}\liminf_{t\rightarrow\infty} P_t(x, B(z,\epsilon))>0.
	\end{align}
	Fixed $\epsilon>0$, from \eqref{C2}, we set $\theta:=\liminf\limits_{t\rightarrow\infty}P_t(z, B(z,\frac\epsilon2))>0$. Similarily, let $f(x)=(1-\frac1\epsilon d(x,z))\lor 0$, then $||f||_{\infty}=1$, $||f||_{Lip}\le \frac1\epsilon$. Moreover, $\frac12 \mathbf{1}_{B(z,\frac\epsilon2)}\le f\le\mathbf{1}_{B(z,\epsilon)}$. Since $\{P_t\}$ is eventually continuous at $z$, we may choose $\delta>0$, for any $y\in B(z,\delta)$,
	\begin{align*}
		\limsup_{t\rightarrow \infty}|P_tf(y)-P_tf(z)|<\frac\theta4.
	\end{align*}
	Then, for any $y\in B(z,\delta)$
	\begin{align}\label{P0}
		\liminf_{t\rightarrow\infty}P_t(y, B(z,\epsilon))>\frac12\liminf_{t\rightarrow\infty}P_t(z, B(z,\frac\epsilon2))-\limsup_{t\rightarrow \infty}|P_tf(y)-P_tf(z)|\geq \frac\theta4.
	\end{align}
	By \textbf{Step 2}, we know that
	\begin{align*}
		\liminf_{t\rightarrow\infty} Q_t(x,B(z,\delta))\geq\mu_*(B(z,\delta)):=\gamma>0.
	\end{align*}
	Hence by the definition of $Q_t(x,B(z,\delta))$, there exists $t_0>0$ such that
	\begin{align*}
		P_{t_0} (x, B(z,\delta))>\gamma>0.
	\end{align*}
	Consequently, using Fatou's lemma
	\begin{align*}
		\liminf_{t\rightarrow\infty}P_t(x,B(z,\epsilon))=&\liminf_{t\rightarrow\infty}P_{t+t_0}(x,B(z,\epsilon))\\
		\geq &\int_{B(z,\delta)}\liminf_{t\rightarrow\infty} P_t(y,B(z,\epsilon)) P_{t_0}(x,\dif y)>\frac{\theta\gamma}4>0.
	\end{align*}
	This implies that for any $\epsilon>0$,
	\begin{align*}
		\inf_{x\in E}\liminf_{t\rightarrow\infty} P_t(x, B(z,\epsilon))>0.
	\end{align*}

	Finally, using the same method as \textbf{Step 2} of the proof for Theorem \ref{thm0} we have for any $x,y\in E$ and $f\in L_b(E)$,
	\begin{align*}
		\limsup_{t\rightarrow \infty}|P_tf(x)-P_tf(y)|=0
	\end{align*}
	is holds.
	Thus, for any $\nu\in\mathcal{M}(E)$, $f\in L_b(E)$ and $y\in E$, using dominated convergence theorem we obtain
	\begin{align*}
		&| \int_{E}f(x)P_t^*\nu(\dif x)-\int_E f(x) \mu(\dif x)|=| \int_{E}P_tf(x)\nu(\dif x)-\int_E P_tf(x) \mu(\dif x)|\\
		\le&  \int_{E}|P_tf(x)-P_tf(y)|\nu(\dif x)+\int_E |P_tf(x)- P_tf(y)| \mu(\dif x)\rightarrow 0, \ \ \text{as} \ \ t\rightarrow \infty.
	\end{align*}
	From \cite[Proposition 2.2]{decom}, asymptotic stability holds. This proof is complete.\qed \bigskip
	
	\subsection{Proof of Corollary \ref{col1}}
	Although the proof follows similar arguments to \cite[Propposition 7.1]{phdth}, we provide the full proof here to highlight our relaxation of conditions.
	
	\emph{Proof of Corollary \ref{col1}.}   For any $r>0$, by using Markov-Chebyshev inequality
	\begin{align*}
		P_t(x, B(z,r))=\EX\mathbf{1}_{d(X_t^x,z)<r}=1-\PX(d(X_t^x, z)\geq r)&\geq 1-\dfrac{\EX[d(X_t^x,z)]^p}{r^p}\\
		&\geq 1-\dfrac{\rho_x(t)}{r^p}-\dfrac{b\mathbf{1}_{B(z,\kappa)}}{r^p}.
	\end{align*}
	Then, there exists a sufficiently large $r_0>0$ such that for all $r\geq r_0$
	\begin{align*}
		\liminf_{t\rightarrow\infty}P_t(x, B(z,r))>\frac12,\ \ \forall x\in E.
	\end{align*}
	For any $r>r_0$ and $\epsilon>0$,  there exists $T=T(\epsilon,r)$  such that
	\begin{align*}
		\inf_{y\in B(z,r)}P_T(y, B(z, \epsilon))>0.
	\end{align*}
	Hence we obtain for any $x\in E$
	\begin{align}\label{coline1}
		\liminf_{t\rightarrow\infty} P_t(x, B(z,\epsilon))=\liminf_{t\rightarrow\infty} P_{T+t}(x, B(z,\epsilon))=&\liminf_{t\rightarrow\infty}\int_E P_t(x, \dif y)P_T(y, B(z,\epsilon)) \nonumber\\
		\geq&\liminf_{t\rightarrow\infty}\int_{B(z, r)} P_t(x, \dif y)\inf_{y\in B(z,r)}P_T(y, B(z,\epsilon))\nonumber\\
		>& \frac12\inf_{y\in B(z,r)}P_T(y, B(z,\epsilon))
	\end{align}
	Combining \eqref{coline1} and eventually continuous at $z$, from Proposition \ref{pro1}, we know that $\{P_t\}_{t\geq 0}$ is asymptotically stable. \qed\bigskip

	\section{Appendix}\label{appendix}
	In this appendix, we present the following claim, which was used in the proof of Theorem \ref{thm2}. The proof of this claim refers to \cite[Proof of Lemma 3]{epro}.

	\noindent\textbf{Claim:}    For given $\epsilon>0,\{\gamma_n\}_{n\in\mathbb{N}^+}, x_{1}, x_{2} \in E$ and $f \in L_b(E)$, the set $\mathcal{D} \neq \varnothing$. Moreover, we have $\sup \mathcal{D}=1$.
	
	\begin{proof}
		First, we show that $\mathcal{D} \neq \varnothing$. Since $\{P_{t} \}_{t \geq 0}$ is eventually continuous at $z \in E$, then  there exist $\sigma>0$ and $t_z>0$ such that
		\begin{equation}\label{s27}
			\left|P_{t} f(z)-P_{t} f(y)\right|<\epsilon / 2 \quad \text { for } y \in B(z, \sigma) \text { and } t \geq t_z .
		\end{equation}
		By \eqref{QC0}, there exist $\zeta>0$ and $T_{0}>0$ such that
		\begin{equation}\label{s28}
			Q_{T}\left(x_{i}, B(z, \sigma)\right) \geq \zeta \quad \forall T \geq T_{0},\  i=1,2 .
		\end{equation}
		Let $\lambda:=\zeta$ and $T_{\lambda, n}=n+T_{0}$ for $n \in \mathbb{N}^+, \mu_{\lambda, i}^{n}(\cdot):=Q_{T_{\lambda, n}}\left(x_{i},\cdot \right)$ and $v_{\lambda, i}^{n}(\cdot):=$ $\mu_{\lambda, i}^{n}(\cdot \mid B(z, \sigma))$ for $i=1,2$ and $n \geq 1$. Note that
		$$\mu_{\lambda, i}^{n}(B(z, \sigma))= Q_{T_{\lambda, n}}\left(x_{i},B(z,\sigma) \right)>0.$$
		The measures $v_{\lambda, i}^{n}, i=1,2$, are supported in $B(z, \sigma)$ and, therefore, for all $t \geq t_z$, we have
		\begin{align*}
			& \left|\int_{E} f(x) P_{t}^{*} v_{\lambda, 1}^{n}(\dif  x)-\int_{E} f(x) P_{t}^{*} v_{\lambda, 2}^{n}(\dif  x)\right| =\left|\int_{E} P_{t} f(x) v_{\lambda, 1}^{n}(\dif  x)-\int_{E} P_{t} f(x) v_{\lambda, 2}^{n}(\dif  x)\right| \\
			\leq&\left|\int_{E}\left[P_{t} f(x)-P_{t} f(z)\right] v_{\lambda, 1}^{n}(\dif  x)\right| +\left|\int_{E}\left[P_{t} f(x)-P_{t} f(z)\right] v_{\lambda, 2}^{n}(\dif  x)\right| \stackrel{\eqref{s27}}{<} \epsilon.
		\end{align*}
		Hence,  for any $T> t_z \lor \dfrac{2||f||_{\infty}t_z}{\epsilon} $
		\begin{align}\label{ecQ}
			& \left|\int_{E} f(x) Q_{T}^{*} v_{\lambda, 1}^{n}(\dif  x)-\int_{E} f(x) Q_{T}^{*} v_{\lambda, 2}^{n}(\dif  x)\right| =\left|\frac1T\int_0^T\bigg[\int_{E} P_{t} f(x) v_{\lambda, 1}^{n}(\dif  x)-\int_{E} P_{t} f(x) v_{\lambda, 2}^{n}(\dif  x)\bigg]\dif t\right| \nonumber\\
			\le&\left|\frac1T\int_0^T\int_{E}\left[P_{t} f(x)-P_{t} f(z)\right] v_{\lambda, 1}^{n}(\dif  x)\dif t\right| +\left|\frac1T\int_0^T\int_{E}\left[P_{t} f(x)-P_{t} f(z)\right] v_{\lambda, 2}^{n}(\dif  x)\dif t\right|\nonumber\\
			\le& \left|\frac1T\int_0^{t_z}\int_{E}\left[P_{t} f(x)-P_{t} f(z)\right] v_{\lambda, 1}^{n}(\dif  x)\dif t\right|+\left|\frac1T\int_{t_z}^T\int_{E}\left[P_{t} f(x)-P_{t} f(z)\right] v_{\lambda, 1}^{n}(\dif  x)\dif t\right|\nonumber\\
			&+ \left|\frac1T\int_0^{t_z}\int_{E}\left[P_{t} f(x)-P_{t} f(z)\right] v_{\lambda, 2}^{n}(\dif  x)\dif t\right|+\left|\frac1T\int_{t_z}^T\int_{E}\left[P_{t} f(x)-P_{t} f(z)\right] v_{\lambda, 2}^{n}(\dif  x)\dif t\right|\nonumber\\
			\le & 2\bigg(\dfrac{2||f||_{\infty}t_z}T+ \epsilon\dfrac{T-t_z}{T}\bigg)<4\epsilon.
		\end{align}
		This implies that $\eqref{s24}$ is holds.  Clearly, conditions \eqref{s20}-\eqref{s23} are also satisfied. Thus, $\mathcal{D} \neq \varnothing$.
		
		Next, we show that $\sup \mathcal{D}=1$. By the definition of $\mathcal{D}$ and $\mathcal{D} \neq \varnothing$, we have $0<\sup \mathcal{D}\le 1$. Suppose there exists $0<\lambda_{0}:=$ $\sup \mathcal{D}<1$. Similar to the proof of Lemma 3 in \cite{epro}, we define a sequence $\{\lambda_n\}_{n\in\mathbb{N}^+}\subset \mathcal{D}$, such that $\lim\limits_{n\rightarrow} \lambda_n=\lambda_0$. Accordingly, we can set a sequence of times $\{T_{n}\}_{n\in\mathbb{N}^+}:=\{T_{\lambda_{n}, n}\}_{n\in\mathbb{N}^+}$ and sequences of measures $ \{\mu_{n, i}\}_{n\in\mathbb{N}^+}:=\{\mu_{\lambda_{n}, i}^{n}\}_{n\in\mathbb{N}^+}$, $\{v_{n, i}\}_{n\in\mathbb{N}^+}:=\{v_{\lambda_{n}, i}^{n}\}_{n\in\mathbb{N}^+}$ for $i=1,2$, satisfy \eqref{s21}-\eqref{s24}. And  $\{\mu_{n, i}\}_{n\in\mathbb{N}^+}$, $\{v_{n, i}\}_{n\in\mathbb{N}^+}$ are tight. Without loss of generality, we may assume that the sequences $\{\mu_{n, i}\}_{n\in\mathbb{N}^+},\{v_{n, i}\}_{n\in\mathbb{N}^+}, i=1,2$, are weakly convergent. The sequences
		\begin{equation}\label{s29}
			\bar{\mu}_{n, i}:=\mu_{n, i}-\lambda_{n} v_{n, i}, \quad n \geq 1
		\end{equation}
		are therefore also weakly convergent for $i=1,2$. The assumption that $\lambda_{0}<1$ implies that the respective limits are nonzero measures; we denote them by $\bar{\mu}_{i}$, $i=1,2$, correspondingly. Analogously to (3.20)-(3.21) in \cite{epro}, we may find $N \geq 1$ such that
		\begin{equation}\label{s212}
			\bar{\mu}_{n, i}\left(B\left(y_{i}, r\right)\right)>\frac{s_{0}}{2} \quad \text { and } \quad \lambda_{n}+ \frac{s_{0}\gamma}{8}>\lambda_{0} \ \ \text{for} \ \ n \geq N.
		\end{equation}
		Here, $y_{i} \in \operatorname{supp} \bar{\mu}_{i}, i=1,2$, exists $r>0$ such that
		\begin{equation}\label{s211}
			Q_{T}(y, B(z, \sigma)) \geq \frac\gamma2 \quad \text { for } y \in B\left(y_{i}, r\right) \text { and } i=1,2 ,
		\end{equation}
		and $s_{0}=\min \left\{\bar{\mu}_{1}\left(B\left(y_{1}, r\right)\right), \bar{\mu}_{2}\left(B\left(y_{2}, r\right)\right)\right\}>0$.
		we prove that $\lambda_{0}^{\prime}:=\lambda_{0}+ \frac{s_{0}\gamma}{8}>\lambda_{0}$ also belongs to $\mathcal{D}$, which obviously leads to a contradiction with the hypothesis that $\lambda_{0}=\sup \mathcal{D}$. We construct sequences $\{T_{\lambda_{0}^{\prime}, n}\}_{n\geq 1},\{\mu_{\lambda_{0}^{\prime}, i}^{n}\}_{n\geq 1}$ and $\{v_{\lambda_{0}^{\prime}, i}^{n}\}_{n\geq 1}, i=1,2$, that satisfy conditions \eqref{s21}-\eqref{s24} with $\lambda$ replaced by $\lambda_{0}^{\prime}$.
		Let $\widehat{\mu}_{n}^{i}(\cdot):=\bar{\mu}_{n, i}\left(\cdot \mid B\left(y_{i}, r\right)\right), i=1,2$, be the measure $\bar{\mu}_{n, i}$ conditioned on the respective balls $B\left(y_{i}, r\right), i=1,2$. That is, if $\bar{\mu}_{n, i}\left(B\left(y_{i}, r\right)\right) \neq 0$, then we let
		\begin{equation}\label{s213}
			\widehat{\mu}_{n}^{i}(\cdot):=\frac{\bar{\mu}_{n, i}\left(\cdot \cap B\left(y_{i}, r\right)\right)}{\bar{\mu}_{n, i}\left(B\left(y_{i}, r\right)\right)} ,
		\end{equation}
		while if $\bar{\mu}_{n, i}\left(B\left(y_{i}, r\right)\right)=0$, we just let $\widehat{\mu}_{n}^{i}(\cdot):=\delta_{y_{i}}$. Also, let $\tilde{\mu}_{n}^{i}(\cdot):=\left(Q_{T}^* \bar{\mu}_{n, i}\right) (\cdot \mid B(z, \sigma))$. From the above definition, it follows that
		\begin{equation}\label{s214}
			Q_{T}^* \mu_{n, i} \geq \frac{s_{0} \gamma}{4} \tilde{\mu}_{n}^{i}+\lambda_{n} Q_{T}^* v_{n, i} ,\ \  \text{for}\ \  n \geq N \ \ \text{and}\ \  i=1,2.
		\end{equation}
		For a detailed explanation of this statement, please refer to (3.25)-(3.27) in \cite{epro}.  since
		\begin{align*}
			\left\|Q_{T_{n}}\left(x_{i},\cdot\right)-\mu_{n, i}\right\|_{\mathrm{TV}}<\gamma_n,
		\end{align*}
		this implies
		\begin{align*}
			\left\|Q_{R, T, T_{n}}\left(x_{i},\cdot\right)-Q_{R, T}^* \mu_{n, i}\right\|_{\mathrm{TV}}<\gamma_n, \ \ \forall R>0.
		\end{align*}
		By \cite[Lemma 2]{epro}, we can choose $R_{n}>T_{n}$ such that
		\begin{equation}\label{s218}
			\left\|Q_{R_{n}, T, T_{n}}\left(x_{i},\cdot\right)-Q_{R_{n}}\left(x_{i},\cdot\right)\right\|_{\mathrm{TV}}<\gamma_n-\left\|Q_{R_{n}, T, T_{n}}\left(x_{i}\right)-Q_{R_{n}, T}^* \mu_{n, i}\right\|_{\mathrm{TV}} .
		\end{equation}
		Now we define
		\begin{equation}\label{s219}
			T_{\lambda_{0}^{\prime}, n}:=R_{n}, \quad \mu_{\lambda_{0}^{\prime}, i}^{n}:=Q_{R_{n}}^* Q_{T}^* \mu_{n, i}, \ \
			v_{\lambda_{0}^{\prime}, i}^{n}:=\dfrac{ Q_{R_{n}}^*\left(\lambda_{n} Q_{T}^* v_{n, i}+\frac{s_{0} \gamma}{4} \tilde{\mu}_{n, i}\right)}{\frac{s_0\gamma}{4}+\lambda_n}
		\end{equation}
		for $i=1,2, n \geq 1$. By virtue of \eqref{s218}, we immediately see that
		$$
		\left\|Q_{T_{\lambda_{0}^{\prime}, n}}\left(x_{i},\cdot\right)-\mu_{\lambda_{0}^{\prime}, i}^{n}\right\|_{\mathrm{TV}}<\gamma_n \quad \forall n \geq 1.
		$$
		Furthermore, from \eqref{s214}, positivity of $Q_{R_{n}}$ and the definitions of $\lambda_{0}^{\prime}$ and measures $\mu_{\lambda_{0}^{\prime}, i}^{n},\  v_{\lambda_{0}^{\prime}, i}^{n}$, we obtain that
		$$
		\mu_{\lambda_{0}^{\prime}, i}^{n} \geq \lambda_{0}^{\prime} \nu_{\lambda_{0}^{\prime}, i}^{n} \quad \forall n \geq N, i=1,2,
		$$
		when $N$ is chosen sufficiently large. To verify \eqref{s24}, note that from \eqref{s219}, it follows that for all $S \geq 0$
		\begin{align}\label{s221}
			&\left|\int_{E} f(x) Q_{S}^* v_{\lambda_{0}^{\prime}, 1}^{n}(\dif  x)-\int_{E} f(x) Q_{S}^* v_{\lambda_{0}^{\prime}, 2}^{n}(\dif  x)\right| \nonumber \\
			\leq&  \frac{\lambda_{n}}{\left(\lambda_{n}+\frac{s_{0} \gamma}{4}\right)}\left|\int_{E} f(x) Q_{S, R_{n}, T}^* v_{n, 1}(\dif  x)-\int_{E} f(x) Q_{S, R_{n}, T}^* v_{n, 2}(\dif  x)\right|\nonumber \\
			& +\dfrac{\frac{s_{0} \gamma}{4}}{\left(\lambda_{n}+\frac{s_{0} \gamma}{4}\right)} \left| \int_{E} f(x) Q_{S, R_{n}}^* \tilde{\mu}_{n, 1}(\dif  x) \dif  s -\int_{E} f(x) Q_{S, R_{n}}^* \tilde{\mu}_{n, 2}(\dif  x)\right|\nonumber\\
			=:&\frac{\lambda_{n}}{\left(\lambda_{n}+\frac{s_{0} \gamma}{4}\right)}I(S)+\dfrac{\frac{s_{0} \gamma}{4}}{\left(\lambda_{n}+\frac{s_{0} \gamma}{4}\right)}II(S).
		\end{align}
		For $I(S)$, we use \cite[Lemma 2]{epro} and \eqref{s24}, which holds for $v_{n, i}, i=1,2$, we then obtain
		\begin{align*}
			\underset{S \rightarrow \infty}{\limsup } I(S) \leq & \underset{S \rightarrow \infty}{\limsup }\left|\int_{E} f(x) Q_{S, R_{n}, T}^* v_{n, 1}(\dif  x)-\int_{E} f(x) Q_{S}^* v_{n, 1}(\dif  x)\right| \\
			& +\limsup _{S \rightarrow \infty}\left|\int_{E} f(x) Q_{S}^* v_{n, 1}(\dif  x)-\int_{E} f(x) Q_{S}^* v_{n, 2}(\dif  x)\right| \\
			& +\limsup _{S \rightarrow \infty}\left|\int_{E} f(x) Q_{S, R_{n}, T}^* v_{n, 2}(\dif  x)-\int_{E} f(x) Q_{S}^* v_{n, 2}(\dif  x)\right|<\epsilon
		\end{align*}
		On the other hand, we note that $\operatorname{supp} \tilde{\mu}_{n}^{i} \subset B(z, \sigma), i=1,2$,  by using \eqref{s27}, \eqref{ecQ} and dominated convergence theorem
		\begin{align*}
			\limsup_{S\rightarrow\infty}    II(S)=\limsup_{S \rightarrow \infty}\left|  \int_{E} \int_{E}\left(Q_{S, R_n} f(x)-Q_{S, R_n} f\left(x^{\prime}\right)\right)  \tilde{\mu}_{n, 1}(\dif  x) \tilde{\mu}_{n, 2}\left(\dif  x^{\prime}\right) \right| =0.
		\end{align*}
		This implies \eqref{s24} holds for $v_{\lambda_{0}^{\prime}, i}^{n}, i=1,2$. Thus, we have shown that $\lambda_0<\lambda_{0}^{\prime} \in \mathcal{D}$. This contradicts $\lambda_{0}=\sup \mathcal{D}$. Therefore, we conclude that $\sup \mathcal{D}=1$.\qed\bigskip\end{proof}
	
	\section*{Acknowledgment}
The authors are sincerely grateful to  Dr. Ziyu Liu for pointing out the reference~\cite{Gong3} to us. This work was supported by NNSF of China (No. 11971186).

\end{document}